\documentclass[11pt,a4paper]{amsart}

\usepackage{amssymb, mathrsfs, amscd}
\usepackage[all]{xy}

\usepackage{graphics}
\usepackage{graphicx}
\usepackage{enumerate}
\usepackage{color}
\usepackage{a4wide}

\newtheorem*{mainthm}{Main Theorem}

\newtheorem*{cor3}{Corollary 3}
\newtheorem*{cor2}{Corollary 2}
\newtheorem*{cor1}{Corollary 1}

\newtheorem{theorem}{Theorem}[section]

\newtheorem{lemma}[theorem]{Lemma}

\newtheorem{proposition}[theorem]{Proposition}
\newtheorem{corollary}[theorem]{Corollary}

\newtheorem{remark}[theorem]{Remark}

\title{A cohomological characterization of locally virtually cyclic groups}
\date{\today}

\author{Dieter Degrijse}
\address{School of Mathematics, Statistics and Applied Mathematics, NUI Galway, Ireland}
\email{dieter.degrijse@nuigalway.ie}%
\newcommand{\mF}{\mathcal {F}}

\newcommand{\Z}{\mathbb Z}

\newcommand{\orb}{\mathcal{O}_{\mF}G}
\newcommand{\orbmod}{\mbox{Mod-}\mathcal{O}_{\mF}G}

\medskip

\begin{document}

\maketitle
\begin{abstract} \noindent
	We show that a countable group is locally virtually cyclic if and only if its  Bredon cohomological dimension for the family of virtually cyclic subgroups is at most one.

\end{abstract}

\section{Introduction}
Let $G$ be a discrete group and recall that the geometric dimension $\mathrm{gd}(G)$ of $G$ is the smallest possible dimension of a free contractible $G$-CW-complex $EG$.  The cohomological dimension $\mathrm{cd}(G)$ of $G$ is by definition the shortest length of a projective resolution of $\mathbb{Z}$ over the group ring $\mathbb{Z}[G]$. The cohomological dimension is the algebraic counterpart of the geometric dimension and satisfies $\mathrm{cd}(G)\leq \mathrm{gd}(G)\leq \max\{3,\mathrm{cd}(G)\}$ by a classic construction of Eilenberg and Ganea (\cite{EilenbergGanea}). One clearly has $\mathrm{gd}(G)=0$ if and only if $G$ is trivial and it is a standard exercise to check that this is equivalent with $\mathrm{cd}(G)=0$. An important and well-known result by Stallings-Swan  (\cite{Stallings},\cite{Swan}) says that
\[ \mathrm{cd}(G)=1 \Longleftrightarrow \mathrm{gd}(G)=1, \] which is equivalent with $G$ being  a non-trivial free group. Hence, the only possible difference between $\mathrm{cd}(G)$ and $\mathrm{gd}(G)$ is that there might exist a group with $\mathrm{cd}(G)=2$ but $\mathrm{gd}(G)=3$. It is a long-standing open question, known as the Eilenberg-Ganea problem, whether or not such a group exists. 

One can generalize the setup above from free actions to actions with stabilizers lying in a certain family of subgroups  $\mathcal{F}$ of $G$, i.e.~a collection of subgroups that is closed under conjugation and taking subgroups. A classifying  $G$-CW-complex for the family $\mathcal{F}$, or model for $E_{\mathcal{F}}G$, is a $G$-$CW$-complex $X$ with stabilizers in $\mathcal{F}$ such that the fixed point sets $X^H$ are contractible for all $H \in \mathcal{F}$. The geometric dimension of $G$ for the family $\mathcal{F}$, denoted by $\mathrm{gd}_{\mathcal{F}}(G)$, is by definition the smallest possible dimension of a model for $E_{\mathcal{F}}G$. This geometric dimension also has an algebraic counterpart: the Bredon cohomological dimension for the family $\mathcal{F}$, denoted by $\mathrm{cd}_{\mathcal{F}}(G)$. The Bredon cohomology of $G$ for the family $\mathcal{F}$ is a generalization of the ordinary group cohomology of G, where the
category of $\mathbb{Z}[G]$-modules is replaced by the category $\mathrm{Mod-}\mathcal{O}_{\mathcal{F}}G$ of contravariant functors
from the orbit category $\mathcal{O}_{\mathcal{F}}G$ to the category of abelian groups. The algebraic invariant $\mathrm{cd}_{\mathcal{F}}(G)$ is defined as the shortest length of a projective resolution of the constant functor $\underline{\mathbb{Z}}$ in the category $\mathrm{Mod-}\mathcal{O}_{\mathcal{F}}G$. It is easy to check that if  one lets $\mathcal{F}$ be the family containing only the trivial subgroup of $G$, then this setup coincides with the free case described in the first paragraph. In \cite[Th. 0.1.]{LuckMeintrup} it is shown that 
\[   \mathrm{cd}_{\mathcal{F}}(G)\leq  \mathrm{gd}_{\mathcal{F}}(G) \leq \max\{3, \mathrm{cd}_{\mathcal{F}}(G)\}.    \]
Moreover, it is easy to see that $\mathrm{gd}_{\mathcal{F}}(G)=0$ if and only if $G \in \mathcal{F}$, and the fact that this is equivalent with  $\mathrm{cd}_{\mathcal{F}}(G)=0$ follows from a result of Symonds (\cite[Lemma 2.5]{Symonds}). It seems out of reach at the moment to say anything more about the relationship between $\mathrm{cd}_{\mathcal{F}}(G)$ and  $\mathrm{gd}_{\mathcal{F}}(G)$ for a general family $\mathcal{F}$. But one can say more when one considers either the family of finite subgroups $\mathcal{FIN}$ or the family of virtually cyclic subgroup $\mathcal{VC}$. Besides the trivial family, the family of finite subgroups and the family of virtually cyclic subgroups have been studied extensively over the last years. This is partially due to their appearance in the Baum-Connes and Farrell-Jones isomorphism conjectures. These conjectures predict isomorphisms between certain equivariant cohomology theories of classifying spaces for families of $G$ and $K$- and $L$-theories of reduced group $C^*$-algebras and group rings of $G$ (e.g. see \cite{LuckReich}). 

When $\mathcal{F}=\mathcal{FIN}$, the numbers $\mathrm{cd}_{\mathcal{F}}(G)$ and $\mathrm{gd}_{\mathcal{F}}(G)$ are denoted by $ \underline{\mathrm{cd}}(G)$ and $\underline{\mathrm{gd}}(G)$, respectively. Note that $\mathrm{cd}_{\mathbb{Q}}(G)\leq \underline{\mathrm{cd}}(G)$, since evaluating a projective $\mathcal{O}_{\mathcal{FIN}}G$-resolution of $\underline{\mathbb{Z}}$ at $G/e$ and tensoring the result with $\mathbb{Q}$ yields a projective $\mathbb{Q}[G]$-resolution of $\mathbb{Q}$. If follows from the work of Dunwoody in \cite{Dunwoody} that
\[ \underline{\mathrm{cd}}(G)=1 \Longleftrightarrow  \underline{\mathrm{gd}}(G)=1\] and that all finitely generated groups with $\underline{\mathrm{cd}}(G)\leq 1$ are virtually free. In this context the answer to the Eilenberg-Ganea problem is also known. Indeed, in \cite{BradyLearyNucinkis} Brady, Leary and Nucinkis construct certain right angled Coxeter groups $G$ with $\underline{\mathrm{cd}}(G)=2$ but $\underline{\mathrm{gd}}(G)=3$. 

Finally, we turn to the family of virtually cyclic subgroups $\mathcal{F}=\mathcal{VC}$. In this case $\mathrm{cd}_{\mathcal{F}}(G)$ and  $\mathrm{gd}_{\mathcal{F}}(G)$ are denoted by $ \underline{\underline{\mathrm{cd}}}(G)$ and $\underline{\underline{\mathrm{gd}}}(G)$, respectively. Here the answer to the Eilenberg-Ganea problem is also known: in \cite{FluchLeary} Fluch and Leary show that certain right angled Coxeter groups $G$ satisfy $\underline{\underline{\mathrm{cd}}}(G)=2$ but $\underline{\underline{\mathrm{gd}}}(G)=3$. In the same paper, Fluch and Leary also ask if $\underline{\underline{\mathrm{cd}}}(G)=1$ if and only if $\underline{\underline{\mathrm{gd}}}(G)=1$. We prove the following.

\begin{mainthm} If a group $G$ satisfies $\underline{\underline{\mathrm{cd}}}(G)\leq 1$, then it is locally virtually cyclic.
\end{mainthm}
The proof of the Main theorem will be finalized in Section 5. By a result of Serre \cite[Cor. 2 p 64]{SerreTrees}, $\underline{\underline{\mathrm{gd}}}(G)\leq 1$ implies that $G$ is locally virtually cyclic. On the other hand, all countable locally virtually cyclic groups $G$ satisfy $\underline{\underline{\mathrm{gd}}}(G)\leq 1$ by \cite[Th.~5.33]{LuckWeiermann}. Hence, our Main Theorem implies the following. 
\begin{cor1}For any countable group $G$, one has 
\[\underline{\underline{\mathrm{cd}}}(G)=1 \Longleftrightarrow  \underline{\underline{\mathrm{gd}}}(G)=1.
\]
\end{cor1}

Recall that a right $\mathcal{O}_{\mathcal{VC}}(G)$-module $M$ is called flat if the categorical tensor product $M\otimes_{\mathcal{O}_{\mathcal{VC}}G}-$ is an exact functor from the category of left $\mathcal{O}_{\mathcal{VC}}(G)$-module to abelian groups. By definition, the constant functor $\underline{\mathbb{Z}}$ is flat if and only if the Bredon homological dimension for the family of virtually cyclic subgroups $\underline{\underline{\mathrm{hd}}}(G)$ equals zero. Since $\underline{\underline{\mathrm{hd}}}(G)\leq \underline{\underline{\mathrm{cd}}}(G)\leq\underline{\underline{\mathrm{hd}}}(G)+1$, when $G$ is countable (see \cite[Prop. 3.5]{nucinkis} and \cite[Section 3.4]{FluchThesis}) and locally virtually cyclic groups $G$ satisfy $\underline{\underline{\mathrm{hd}}}(G)=0$ by \cite[Prop 3.1(i)]{FluchNucinkis}, the Main Theorem implies the following.

\begin{cor2} A countable group $G$ is locally virtually cyclic if and only if the constant functor $\underline{\mathbb{Z}}$ is a flat $\mathcal{O}_{\mathcal{VC}}(G)$-module.

\end{cor2}
In light of Corollary 2, we cannot resist mentioning a conjecture of L\"{u}ck (\cite[ 6.49]{LuckL2Book}) saying that a countable group $G$ is locally virtually cyclic if and only if the group von Neumann algebra $\mathcal{N}(G)$ is a flat $\mathbb{C}[G]$-module. \\

Let $F_2$ be the free group on $2$ generators and recall that a group is said to be coherent if all its finitely generated subgroups are finitely presented. As a by-product of our proof of the Main Theorem, we can also show that certain groups with $\underline{\mathrm{cd}}(G)=2$ admit a three dimensional model for $\underline{\underline{E}}G$.
\begin{cor3} Let $G$ be a countable group satisfying  $\underline{\mathrm{cd}}(G)=2$. If one of the following hold
\begin{itemize}
\item[(a)] $G$ is coherent
\item[(b)] $G$ is torsion-free
\item[(c)] $G$ does not contain $F_2$,
\end{itemize}
then $1 \leq \underline{\underline{\mathrm{cd}}}(G)\leq 3$.  Moreover, every one-relator group has a three dimensional model for its classifying space with virtually cyclic stabilizers.

\end{cor3}
\noindent The proof of this corollary will be given at the end of Section 4. Note that it follows from \cite[Prop. 5.1]{DP} and \cite[Corollary 4]{DP2} that the bounds for $\underline{\underline{\mathrm{cd}}}(G)$ given above are optimal and that all intermediate values can occur.
Motivated by Remark \ref{remark}, we conjecture that $\underline{\underline{\mathrm{gd}}}(G)\leq 3$ for all countable groups with $\underline{\mathrm{cd}}(G)=2$.

\section{Preliminaries}
We start by recalling some basic notions in Bredon cohomology. This cohomology theory was introduced by Bredon in \cite{Bredon} for finite groups as a mean to develop an obstruction theory for equivariant extension of maps. It was later generalized to arbitrary groups by L\"{u}ck with applications to finiteness conditions  (see \cite[section 9]{Luck}). We refer the reader to \cite{FluchThesis} and \cite{Luck} for further details. Let $G$ be a discrete group and let $\mathcal{F}$ be a family of subgroups of $G$. The orbit category $\orb$ is a category whose objects that are the left cosets $G/H$ for all $H \in \mathcal{F}$ and whose morphisms are all $G$-equivariant maps between the objects.
A (right) $\orb$-module is a contravariant functor $M: \orb \rightarrow \mathbb{Z}\mbox{-mod}$. The category of $\orb$-modules is denoted by $\orbmod$ and is defined by the objects that are all the $\orb$-modules and the morphisms are all the natural transformations between the objects.
This category is abelian and has free objects. More precisely, its free objects are isomorphic to direct sums of modules of the form $\Z[-,G/K]$
where $K\in\mathcal{F}$ and $\Z[G/H,G/K]$ is the free $\Z$-module with basis the set of $G$-maps from $G/H$ to $G/K$. A sequence of modules in $\orbmod$ is said to be exact if it is exact when evaluated at every object. There is also a Yoneda lemma that allows one to construct free (projective) resolutions in a similar way as in the case of ordinary cohomology. It follows that $\orbmod$ contains enough projectives to construct projective resolutions. Hence, one can construct Ext-functors $\mathrm{Ext}^{n}_{\orb}(-,M)$ that have all the usual properties. If $V$ is a $G$-module, we denote by \[\underline{V}: \mathcal{O}_{\mathcal{F}}G \rightarrow \mathbb{Z}\mbox{-mod}\] the contravariant functor that maps $G/H$ to the fixed points $V^H$ and maps morphisms to the obvious restriction maps. The $n$-th Bredon cohomology of $G$ with coefficients $M \in \orbmod$ is by definition
\[ \mathrm{H}^n_{\mathcal{F}}(G,M)= \mathrm{Ext}^{n}_{\orb}(\underline{\mathbb{Z}},M) \]
where $\mathbb{Z}$ is a trivial $G$-module. Low-dimensional Bredon cohomology with fixed point functor coefficients can be understood in terms of ordinary group cohomology.
\begin{lemma}{\cite[Cor. 5.6, 5.14]{DP3}} \label{lem: fixed point} For any $G$-module $V$ one has $$\mathrm{H}^0_{\mathcal{F}}(G,\underline{V})=V^G$$
	and
	\[     \mathrm{H}^1_{\mathcal{F}}(G,\underline{V})=\bigcap_{K \in \mathcal{F}} \ker :  \mathrm{H}^1(G,V) \rightarrow  \mathrm{H}^1(K,V).  \]
Moreover, if $\mathrm{H}^1(K,V)=0$ for all $K \in \mathcal{F}$, then 	
	\[  \mathrm{H}^2_{\mathcal{F}}(G,\underline{V})=\bigcap_{K \in \mathcal{F}} \ker :  \mathrm{H}^2(G,V) \rightarrow  \mathrm{H}^2(K,V).  \]

\end{lemma}
The  Bredon cohomological dimension of $G$ for the family $\mathcal{F}$, denoted by $\mathrm{cd}_{\mathcal{F}}(G)$, is defined as
\[ \mathrm{cd}_{\mathcal{F}}(G) = \sup\{ n \in \mathbb{N} \ | \ \exists M \in \orbmod :  \mathrm{H}^n_{\mathcal{F}}(G,M)\neq 0 \}. \]
By standard methods in homological algebra one shows that this number coincides with the length of the shortest projective (or free) resolution of $\underline{\mathbb{Z}}$ in  $\orbmod$. Using a version of Shapiro's lemma for Bredon cohomology, one can show that $$\mathrm{cd}_{\mathcal{F}\cap S}(S) \leq  \mathrm{cd}_{\mathcal{F}}(G)$$ for any subgroup $S$ of $G$. \\

For the remainder of this paper, $\mathcal{F}$ will always denote the family of finite subgroups.\\

A key ingredient in the proof of the Main Theorem is the following construction of L\"{u}ck and Weiermann (see \cite{LuckWeiermann}) which relates Bredon cohomology for the family of virtually cyclic subgroups $\mathcal{VC}$ to Bredon cohomology for the family of finite subgroups $\mathcal{F}$. Let $\mathcal{S}$ denote the set of infinite virtually cyclic subgroups of $G$. As in   \cite[2.2]{LuckWeiermann}, two infinite virtually cyclic subgroups $H$ and $K$ of $G$ are said to be equivalent, denoted $H \sim K$, if \[|H\cap K|=\infty\]
i.e. if $H$ and $K$ are commensurable, meaning that their intersection has finite index in both $H$ and $K$. Using the fact that any two infinite virtually cyclic subgroups of a virtually cyclic group are equivalent (e.g.~see Lemma 3.1.~in \cite{DP1}), it is easily seen that this indeed defines an equivalence relation on $\mathcal{S}$. One can also verify that this equivalence relation satisfies the following two properties
\begin{itemize}
	\item[-]  $\forall H,K \in \mathcal{S} : H \subseteq K \Rightarrow H \sim K$;
	\item[-] $ \forall H,K \in \mathcal{S},\forall g \in G: H \sim K \Leftrightarrow H^g \sim K^g$.
\end{itemize}
Let $H \in \mathcal{S}$ and define the group
\[     \mathrm{N}_G[H]=\{g \in G \ | \ H^g \sim H \}. \]
The group $ \mathrm{N}_G[H]$ is called the commensurator of $H$ in $G$. It is sometimes denoted by $\mathrm{Comm}_G(H)$ and should not be confused with the normalizer of $H$ in $G$. Note that $\mathrm{N}_G[H]$ depends only on the equivalence class $[H]$ of $H$ and contains $H$ as a subgroup. Define for $H \in \mathcal{S}$ the following family of subgroups of $\mathrm{N}_G[H]$
\[ \mathcal{F}[H]=\{ K \subseteq \mathrm{N}_G[H] \ | K \in \mathcal{S}, K \sim H\} \cup \Big(\mathrm{N}_G[H] \cap \mathcal{F}\Big). \]
In other words,  $\mathcal{F}[H]$ contains all finite subgroups of $ \mathrm{N}_G[H]$ and all infinite virtually cyclic subgroups of $G$ that are equivalent to $H$. The pushout diagram in Theorem $2.3$ of \cite{LuckWeiermann} now yields the following (see also \cite[\S 7]{DP}).
\begin{proposition}[L\"{u}ck-Weiermann, \cite{LuckWeiermann}] \label{prop: long exact} \label{prop: push out} With the notation above, let $[\mathcal{S}]$ denote the set of equivalence classes of $\mathcal S$   and let $\mathcal{I}$ be a complete set of representatives $[H]$ of the orbits of the conjugation action of $G$ on $[\mathcal{S}]$. For every $M \in \mbox{Mod-}\mathcal{O}_{\mathcal{VC}}G$, there exists a long exact sequence
	\[ \ldots \rightarrow \mathrm{H}^{i}_{\mathcal{VC}}(G,M) \rightarrow \Big(\prod_{[H] \in \mathcal{I}} \mathrm{H}^{i}_{\mathcal{F}[H] }( \mathrm{N}_G[H],M)\Big)\oplus  \mathrm{H}^{i}_{\mathcal{F}}(G,M) \rightarrow  \prod_{[H] \in \mathcal{I}} \mathrm{H}^{i}_{\mathcal{F} }( \mathrm{N}_G[H],M)  \]
	\[ \rightarrow \mathrm{H}^{i+1}_{\mathcal{VC}}(G,M) \rightarrow \ldots \ . \]
\end{proposition}

We end this section with some simple observations about groups with $\underline{\underline{\mathrm{cd}}}(G)=1$. Recall that the Baumslag-Solitar group $\mathrm{BS}(1,m)$, with $m \in \mathbb{Z}\setminus\{0\}$, is the group with presentation
\[      \mathrm{BS}(1,m)=\langle t,x \ | \ txt^{-1}=x^m \rangle.   \]
 This group is solvable and is by definition and ascending HNN-extension of $\mathbb{Z}$.
\begin{lemma}\label{lem: basic} If $G$ is a group with $\underline{\underline{\mathrm{cd}}}(G)\leq 1$, then the following hold.
\begin{itemize}
\item[(1)] The group $G$ does not contain a copy of a Baumslag-Solitar group $\mathrm{BS}(1,m)$ for any $m \in \mathbb{Z}\setminus\{0\}$.
\item[(2)]  The group $G$ does not contain a copy of the free group on two generators 	$F_2$.
\item[(3)] For every right $\mathcal{O}_{\mathcal{VC}}(G)$-module $M$, there is an isomorphism 

\begin{equation*} 
\prod_{[H]\in \mathcal{I}}\mathrm{H}^2_{\mathcal{F}[H]}(\mathrm{N}_G[H],M)\oplus  \mathrm{H}^2_{\mathcal{F}}(G,M) \xrightarrow{\cong}  \prod_{[H]\in \mathcal{I}}\mathrm{H}^2_{\mathcal{F}}(\mathrm{N}_G[H],M).
\end{equation*}	
\end{itemize}
	
\end{lemma}
\begin{proof} Since $\underline{\underline{\mathrm{cd}}}(\mathrm{BS}(1,m))\geq 2$ by \cite[Corollary 3]{DP2} and $\underline{\underline{\mathrm{cd}}}(F_2)=2$ by \cite[Lemma 5.2]{DP2}, it follows that $G$ cannot contain copies of these groups. This proves (1) and (2). Part (3) is immediate from Proposition \ref{prop: long exact}.
	
\end{proof}
\begin{lemma} \label{lemma: ration dim} If $G$ is a finitely generated group with $\underline{\underline{\mathrm{cd}}}(G)= 1$, then $\mathrm{cd}_{\mathbb{Q}}(G)=2$ and 
	\[  \mathrm{H}^1(G,F)=0. \]
	for any free $\mathbb{Q}[G]$-module $F$.
\end{lemma}
\begin{proof} It follows from \cite[Cor. 4.2]{DemPetTal} that $\mathrm{cd}_{\mathbb{Q}}(G)\leq \underline{\mathrm{cd}}(G)\leq 2$. If $\mathrm{cd}_{\mathbb{Q}}(G)\leq 1$, then $G$ is virtually free by Dunwoody's result mentioned in the introduction. However, since $G$ cannot contain $F_2$ by Lemma \ref{lem: basic}(2), we conclude that $G$ is virtually cyclic, implying that $\underline{\underline{\mathrm{cd}}}(G)= 0$. This is impossible so we conclude that $\mathrm{cd}_{\mathbb{Q}}(G)=2$. Since $G$ cannot contain $F_2$ and is not virtually cyclic, it follows from Stalling's theorem about ends of groups that $G$ is one-ended, i.e. $\mathrm{H}^1(G,\mathbb{Q}[G])=0$. Because $G$ is finitely generated, this implies that $\mathrm{H}^1(G,F)=0$
	for any free $\mathbb{Q}[G]$-module $F$.
\end{proof}
\section{relative ends}
Let $G$ be a finitely generated group and let $H$ be an infinite index finitely generated subgroup of $G$. For any commutative ring $R$ with unit, one defines $\mathrm{E}_H^G(R)$ to be the $G$-fixed points of the cokernel of the natural inclusion of right $G$-modules
\[ \mathcal{F}_H^G(R)= \mathrm{ind}_H^G (\mathrm{coind}_e^H (R))\rightarrow \mathrm{coind}_H^G (\mathrm{coind}_e^H(R)) =  \mathrm{coind}_e^G (R) .     \]
By considering the long exact cohomology sequence associated to the short exact sequence of right $G$-modules 
\[   0\rightarrow \mathcal{F}_H^G(R) \rightarrow \mathrm{coind}_e^G(R) \rightarrow  \mathrm{coind}_e^G(R)/\mathcal{F}_H^G(R) \rightarrow 0, \]
an application of Shapiro's lemma yields  the exact sequence
\begin{equation}  \label{eq: ends exact} 0 \rightarrow R \rightarrow E_H^G(R) \xrightarrow{\delta} \mathrm{H}^1(G,\mathcal{F}_H^G(R)) \rightarrow 0.  \end{equation}
In \cite{KrophollerRoller}, Kropholler and Roller define the number of ends of the pairs $(G,H)$ to be the number
\[   \tilde{e}(G,H)= \dim_{\mathbb{F}_2}\Big(E_H^G(\mathbb{F}_2)\Big).    \]
where $\mathbb{F}_2$ is the field with two elements. Note that if $H$ equals the trivial subgroup of $G$, then $\tilde{e}(G,H)$ coincides with the number of ends of $G$, mentioned in the previous section.

A subset $A$ of $G$ is said be $H$-finite (resp.~$H$-cofinite) if it (resp.~its complement) can be covered by finitely many right cosets of $H$. The subset $A$ is said to be proper if it is neither $H$-finite nor $H$-cofinite. A subset $B$ of $G$ is called $H$-almost invariant if the symmetric difference $B\Delta (Bg)$ is $H$-finite for all $g \in G$. As we shall see later on, the importance of proper $H$-almost invariant subsets of $G$ is contained in the fact that under certain additional hypotheses they can be used to construct non-trivial $G$-actions on trees.

Note that $\mathrm{coind}_e^G(\mathbb{F}_2)$ corresponds to the set of subsets of $G$ equipped with symmetric difference as group-law. One easily checks that the $H$-finite subsets of $G$ are exactly the elements of $\mathcal{F}_H^G(\mathbb{F}_2)$, while the $H$-cofinite subsets of $G$ correspond to elements $A$ of $\mathrm{coind}_e^G(\mathbb{F}_2)$ such that $G+A \in \mathcal{F}_{H}^G(\mathbb{F}_2)$. Finally, one verifies that the proper $H$-almost invariant subsets $B$ of $G$ correspond to the elements of $E_H^G(\mathbb{F}_2)$ that map to a non-zero element under the map $\delta$ appearing in (\ref{eq: ends exact}). These observations lead to the following lemma.

\begin{lemma}\cite[Lemma 3.1]{KrophollerRoller}\label{lemma : almost invariant} If $\tilde{e}(G,H)\geq 2$ and $G=\mathrm{N}_G[H]$, then $G$ has a proper $H$-almost invariant subset $B$ such that $BH=B$.
	
\end{lemma}
\begin{proof} Consider the commutative diagram with exact rows
	\[\xymatrix{ \mathbb{F}_2 \ar[r] \ar[d] & E_H^G(\mathbb{F}_2) \ar[r]^{\delta} \ar[d]\ar[r] & \mathrm{H}^{1}(G,\mathcal{F}_H^G(\mathbb{F}_2))  \ar[d]   \\  (\mathrm{coind}_H^G(\mathbb{F}_2))^H \ar[r]^{\alpha \ \ \ \ \ \ \ }  &  \Big( \mathrm{coind}_e^G(\mathbb{F}_2)/\mathcal{F}_H^G(\mathbb{F}_2) \Big)^H \ar[r]  & \mathrm{H}^{1}(H,\mathcal{F}_H^G(\mathbb{F}_2))   }	\]
	where the vertical maps are restriction maps induced by the inclusion $H\subseteq G$. The double coset formula says that $\mathcal{F}_H^G(\mathbb{F}_2)$ restricted to $H$ is isomorphic to
	\[   \bigoplus_{g\in [H \backslash G / H]} \mathrm{ind}_{H\cap gHg^{-1}}^{H} \Big(\mathrm{res}_{H\cap gHg^{-1}}^{gHg^{-1}} \big(g\cdot \mathrm{coind}_e^H (\mathbb{F}_2)\big)\Big) .  \]

	Since $\mathrm{N}_G[H]=G$, $H \cap gHg^{-1}$ has finite index in $H$ and $gHg^{-1}$ for all $g\in G$. If follows that, as a $H$-module, $\mathcal{F}_H^G(\mathbb{F}_2)$ is isomorphic to a direct sum of modules coinduced up from the trivial group. Since $H$ is finitely generated, this implies that $\mathrm{H}^{1}(H,\mathcal{F}_H^G(\mathbb{F}_2))=0 $. As $\tilde{e}(G,H)\geq 2 $, there exist an element $B$ of $E_H^G(\mathbb{F}_2)$ such that $\delta(B)\neq 0$, i.e. $B$ is a proper $H$-almost invariant subset of $G$. By the diagram above, we conclude that $B$ must be contained in the image of $\alpha$, which means that $BH=B$.
	
\end{proof}
In the next section, we will apply this lemma. However, due to the presence of torsion in our groups, we will be forced to consider cohomology with coefficient over the rationals. We therefore need to argue that one can also compute $\tilde{e}(G,H)$ over $\mathbb{Q}$ instead of over $\mathbb{F}_2$. To show this we follow the approach of Swan in \cite[Lemma 3.6]{Swan} where this is proven for $H=\{e\}$. \\

Fix a commutative ring $R$ with unit $1 \neq 0$ and a finite generating set $\{\sigma_1,\ldots,\sigma_d\}$ for $G$. 
\begin{lemma}\label{tech lemma}
	Let $F$ be the $R$-module of set maps $f$ from $G$ to $R$ turned into a right $G$-module via $(f\cdot g)(x)=f(x\cdot g^{-1})$. Let $\Psi$ be the subspace of $F$ consisting of those maps from $G$ to $R$ that take non-zero values in at most finitely many right cosets $Hg$ of $H$ in $G$ and let $\Theta$ be the subspace of $F$ consisting of all maps for which there exists a finite union $S=\bigcup_{j=1}^mHg_j$ such that $f( x \sigma_i)=f(x)$ for all $x \in G\setminus S$ and all $i\in \{1,\ldots,d\}$. Then  $$    E_H^G(R) \cong \Theta/\Psi$$ as $R$-modules.
	
\end{lemma}
\begin{proof}
	
Take $f \in \Psi$ and let $S$ be a finite union of right cosets of $H$ such that $f(x)=0$ for all $x \in G \setminus S$. Letting $S'=S\cup \bigcup_{i=1}^d S \sigma_i^{-1}$, one sees that $f(x)=f(x\sigma_i)=0$ for all $x \in G\setminus S'$ and all $i\in \{1,\ldots,d\}$. Indeed, if $f(x\sigma_i)\neq 0$ then $x\sigma_i \in S$ and hence $x\in S\sigma_i^{-1}$. This proves that $\Psi\subseteq \Theta$. Explicitly, the natural inclusion of right $G$-modules
of $\mathcal{F}_H^G(R)$ into  $\mathrm{coind}_e^G (R)$ is given by

	\[  \mathrm{Hom}_R(R[H],R)\otimes_{R[H]}R[G]\rightarrow  \mathrm{Hom}_R(R[G],R): f \otimes g \mapsto \tilde{f}    \]
	where $\tilde{f}(x)=f(gx)$ if $x\in g^{-1}H$ and $\tilde{f}(x)=0$ otherwise. Composing this inclusion with the isomorphism of right $G$-modules
		\[   \mathrm{Hom}_R(R[G],R) \xrightarrow{\cong} F : f  \mapsto \hat{f}    \]
		where $\hat{f}(g)=f(g^{-1})$ for $g\in G$, gives the inclusion
		\[  	  \mathrm{Hom}_R(R[H],R)\otimes_{R[H]}R[G]\rightarrow  F: f \otimes g \mapsto \overline{f}    \]     
	where $\overline{f}(x)=f(gx^{-1})$ if $x\in Hg$ and $\overline{f}(x)=0$ otherwise. This shows that $\mathcal{F}_H^G(R)\cong \Psi $ and hence $E_H^G(R)\cong (F/\Psi)^G$. Using the fact that $\{\sigma_1,\ldots,\sigma_d\}$ is a set of generators for $G$, one verifies that $\Theta$ coincides with the subspace of $F$ consisting of all maps such that for each $g \in G$ there exists a finite union $S=\bigcup_{j=1}^mHg_j$ such that $f( x g)=f(x)$ for all $x \in G\setminus S$. The lemma now follows easily.

\end{proof}

Recall that the rank $\mathrm{rk}(M)$ of $R$-module $M$ is the smallest positive integer $n$ such that $R^n$ admits a surjection onto $M$. If no such integer exists then the rank is infinite.
\begin{proposition} For any commutative ring $R$ with unit $1 \neq 0$, one has
	\[   \tilde{e}(G,H)= \mathrm{rk}\Big(E_H^G(R)\Big) . \]
	
\end{proposition}
\begin{proof}Let $\Gamma$ be the Cayley graph of $G$ for the generating set $\{\sigma_1,\ldots,\sigma_d\}$ such that $G$ acts on $\Gamma$ from the left, i.e. the vertices of $\Gamma$ are the elements of $G$ and $g_1,g_2\in G$ are connected by an edge iff $g_1\sigma_{i}^{\pm 1}=g_2$ for some $i\in \{1,\ldots, d\}$. Since $\{\sigma_1,\ldots,\sigma_d\}$ generates $G$, it follows that $\Gamma$ is connected. Let $S=\bigcup_{i=1}^n Hg_i$ be a union of finitely many right cosets of $H$ in $G$. By $\Gamma \setminus S$ we mean the subgraph of $\Gamma$ spanned by the vertices $G\setminus S$. Since $S$ is $H$-invariant, $H$ acts on $\Gamma \setminus S$ . We will be interested in the connected components of $\Gamma \setminus S$ . Fix a vertex $x_0$ of $\Gamma \setminus S$. Given $x \in \Gamma \setminus S$, there exists a path $\gamma$ in $\Gamma$ from $x_0$ to $x$. If $\gamma$ does not cross $S$ then $x$ and $x_0$ are in the same connected component of $\Gamma \setminus S$. If $\gamma$ does cross $S$ then there exists a vertex $s \in S$ and a generator $\sigma_i$ such that $s\sigma_i^{\pm 1}$ is vertex of $\gamma$ and the subpath of $\gamma$ going from $s\sigma_i^{\pm 1}$ to $x$ does not cross $S$. Hence $s\sigma_{i}^{\pm 1}$ and $x$ are contained in the same connected component of $\Gamma \setminus S$. It follows that every connected component of $\Gamma\setminus S$ contains either $x_0$ or an element of the form $hg_j\sigma_{i}^{\pm 1}$ for some $j \in \{1,\ldots,m\}$ and some $i \in \{1\ldots ,d\}$. We conclude that there are only finitely many $H$-orbits of connected components of $\Gamma \setminus S$.
	
Enumerate the elements of $G=\{e=g_1,g_2,g_3,\ldots\}$ and let $S_{0}=\emptyset$. If $S_n$ is defined, then let $S_{n+1}'=S_n \cup \{Hg_{n+1}\}$ and define $S_{n+1}$ to be the union of $S'_{n+1}$ with all the  
	connected components of $\Gamma \setminus S'_{n+1}$ that can be covered by finitely many right cosets of $H$ in $G$. Note that every $S_n$ is a finite union of right cosets of $H$ in $G$, that $S_{n} \subset S_{n+1}$ and $G=\bigcup_{n\geq 1} S_n$. Now let $C_n$ be the set of components of $\Gamma\setminus S_n$. Since each component of $\Gamma \setminus S_{n+1}$ is contained in a unique component of $G \setminus S_{n}$, we obtain a map $\alpha_{n+1}:C_{n+1}\rightarrow C_n$. By construction, no component of $\Gamma \setminus S_n$ can be covered by finitely many right cosets of $H$ in $G$, so it follows that the map $\alpha_{n+1}:C_{n+1}\rightarrow C_{n}$ is surjective. Indeed, let $X$ be a component of $\Gamma \setminus S_n$. We know that $X$ cannot be covered by finitely many right cosets of $H$ in $G$. If $Hg_{n+1}$ does not intersect $X$, then $X$ is a connected component of $\Gamma \setminus S_{n+1}'$ that cannot be covered by finitely many right cosets of $H$ in $G$. Therefore $X \in C_{n+1}$ and $\alpha_{n+1}(X)=X$. Now suppose that $X$ intersects with $Hg_{n+1}$ and consider the connected components of $X \setminus (Hg_{n+1}\cap X)$. Suppose all these connected components could be covered by finitely many right cosets of $H$ in $G$. Since there are only finitely many $H$-orbits of connected components in $\Gamma \setminus S'_{n+1}$, it follows that $X$ can by covered by finitely many right cosets of $H$ in $G$, which is a contradiction. Hence there must exist a connected component $X_0$ that cannot be finitely covered. By construction, $X_0$ will be a subset of an element of $C_{n+1}$ that maps to $X$ under $\alpha_{n+1}$.
	
	Now let $E_H^G(R)_n$ be the $R$-module of set maps from $C_n$ to $R$. The surjection $C_{n+1}\rightarrow C_n$ defines a monomorphism 
	\[   E_H^G(R)_n\rightarrow E_H^G(R)_{n+1}.      \]
	onto a direct summand of $ E_H^G(R)_{n+1}$. 
	Define for each $n\geq 1$ the map 
	\[\varphi_n : E_H^G(R)_n \rightarrow \Theta/\Psi: f \mapsto \tilde{f}+\Psi\]
	where $\tilde{f}(x)=f(c)$ if $x$ lies in the component $c$ of $\Gamma \setminus S_n$ and $\tilde{f}(x)=0$ if $x \in S_n$.  To see that $\tilde{f} \in \Theta$,  set $S=S_n\cup \bigcup_{i=1}^d S_n\sigma^{-1}_i$ and pick $x \in G\setminus S$. Since $x$ and $x\sigma_i$ are connected by an edge in $\Gamma$, the only way $x$ and $x\sigma_i$ can not be in the same component of $\Gamma \setminus S_n$ is if $x\sigma_i \in S_n$, but then $x \in S_{n}\sigma^{-1}_i$, which is not possible. We conclude that $\tilde{f}(x\sigma_i)=\tilde{f}(x)$ for all $x \in S$ and all $i \in \{1,\ldots,d\}$. One checks that the maps $\varphi_n$ assemble to form an $R$-module  map
	\[   \varphi: \lim_{\rightarrow}  E_H^G(R)_n \rightarrow \Theta/\Psi. \]
	We claim that $\varphi$ is an isomorphism. Indeed, take $f \in \Theta$ and let $S$ be a finite union of right cosets of $H$ in $G$ such that $f( x \sigma_i)=f(x)$ for all $x \in G\setminus S$ and all $i\in \{1,\ldots,d\}$. Then $S\subseteq S_n$ for $n$ large enough, implying that $f$ is constant on the connected components of $\Gamma \setminus S_n$. Hence $f \in E_H^G(R)_n$, proving that $\varphi$ is surjective. Now suppose that $f \in E_H^G(R)_n$ and that $\varphi_n(f)=0$. Then $\tilde{f}(x)=0$ for all $x \in G \setminus S$, where $S$ is a finite union of right cosets of $H$ in $G$.
	Since $S \subseteq S_m$ for some $m\geq n$, it follows that $f \rightarrow 0$ in $E_H^G(R)_{m}$, proving that $\varphi$ is injective. This proves our claim.

	We now consider two cases. First suppose there there exists a $k\geq 0$ such that $C_k$ is infinite. In this case $$\mathrm{rk}_R(E_{H}^G(R)_k)=\mathrm{rk}_R(\prod_{c \in C_k}R)$$ is infinite. Since $E_H^G(R)_k$ is a direct summand of $$\lim_{\rightarrow}  E_H^G(R)_n\cong \Theta/\Psi \cong E_H^G(R)$$ by the above and Lemma \ref{tech lemma}, we conclude that $\mathrm{rk}_R(E_H^G(R))=\infty$ for any commutative ring $R$ with unit $1\neq 0$. This holds in particular for $R=\mathbb{F}_2$, so $\tilde{e}(G,H)=\infty$. Secondly, assume that $C_n$ is finite for all $n\geq 0$. In this case one has $$E_H^G(R)_n \cong R \otimes_{\mathbb{Z}} E_H^G(\mathbb{Z})_n.$$ We conclude by taking limits and Lemma \ref{tech lemma} that 
	$E_H^G(R)=R\otimes_{\mathbb{Z}} E_H^G(\mathbb{Z})$.
	Since  $E_H^G(\mathbb{Z})_{n+1}\cong E_H^G(\mathbb{Z})_{n} \oplus X_{n+1}$ where $X_{n+1}$ is a finitely generated free abelian group, it follows that $E_H^G(\mathbb{Z})$ is a free abelian group of countable rank. We conclude that $$\mathrm{rk}_R(E_H^G(R))=\mathrm{rk}_{\mathbb{Z}}(E_H^G(\mathbb{Z}))$$ for any commutative ring $R$ with unit $1\neq 0$. This holds in particular for $R=\mathbb{F}_2$, so $$\tilde{e}(G,H)=\mathrm{rk}_R(E_H^G(R))$$ as desired.
	
\end{proof}

\begin{corollary}\label{cor: rat end} One has
	\[   \tilde{e}(G,H)=\dim_{\mathbb{Q}}\Big(\mathrm{H}^1(G,\mathcal{F}_H^G(\mathbb{Q}))\Big)+1 . \]
	
\end{corollary}
\begin{proof} This follows immediately from the proposition and (\ref{eq: ends exact}).
\end{proof}
\section{Structure of commensurators}
A crucial ingredient for proving our main theorem is the following variation on a result due to Kropholler.
\begin{theorem}\label{th: kropholler} If a finitely generated group $G$ has rational cohomological dimension two and has an infinite virtually cyclic subgroup $H$ such that $G=\mathrm{N}_G[H]$, then $G$ splits over a virtually cyclic subgroup commensurable with $H$. In particular, $G$ contains $F_2$ or a  Baumslag-Solitar group $\mathrm{BS}(1,m)$ with $m \in \mathbb{Z}\setminus \{0\}$.

\end{theorem}
This theorem allows us to determine the structure of commensurators of infinite virtually cyclic groups in a group $G$ with $\underline{\underline{\mathrm{cd}}}(G)\leq 1$.
\begin{corollary}\label{cor: comm loc cyc} If $G$ is a group with $\underline{\underline{\mathrm{cd}}}(G)\leq 1$, then the commensurator $\mathrm{N}_G[H]$ of any infinite virtually cyclic subgroup $H$ of $G$ is locally virtually cyclic.
	
\end{corollary}
\begin{proof} Let $H$ be an infinite virtually cyclic subgroup of $G$ and let $\Gamma$ be a finitely generated subgroup of $\mathrm{N}_G[H]$ that contains $H$. Then $\Gamma$ is a finitely generated group with $N_{\Gamma}[H]=\Gamma$ and $\underline{\underline{\mathrm{cd}}}(\Gamma)\leq 1$. Our aim is to show that $\Gamma$ is virtually cyclic. Assume that $\Gamma$ is not virtually cyclic. Then $\underline{\underline{\mathrm{cd}}}(\Gamma)= 1$, in which case $\mathrm{cd}_{\mathbb{Q}}(\Gamma)= 2$ by Lemma \ref{lemma: ration dim}. But then it follows from Theorem \ref{th: kropholler} that $\Gamma$ contains $F_2$ or $\mathrm{BS}(1,m)$ which implies that $\underline{\underline{\mathrm{cd}}}(\Gamma)\geq 2$ by  Lemma \ref{lem: basic}(2). This is a contradiction, so we conclude that $\Gamma$ is virtually cyclic.
	
\end{proof}
We now turn to the proof of Theorem \ref{th: kropholler}. In \cite{kropholler}, Kropholler proves that if a torsion-free finitely generated group $G$ has cohomological dimension two and has a infinite cyclic subgroup $H$ such that $G=\mathrm{N}_G[H]$, then $G$ splits over an infinite cyclic subgroup commensurable with $H$. Below we will show that Kropholler's arguments remain valid for non-torsion-free groups with rational cohomological dimension equal to two. \\

Let $G$ be a finitely generated group that has rational cohomological dimension equal to two and has an infinite cyclic subgroup $H$ such that $G=\mathrm{N}_G[H]$.

\begin{lemma}\cite[Lemma 2.2]{kropholler} The group $G$ is a $2$-dimensional duality group over $\mathbb{Q}$.
	
\end{lemma}
\begin{proof}
Strebel's criterion mentioned in the beginning of \cite[Section 2]{kropholler} and \cite[Lemma 2.1]{kropholler} are valid over any field $k$. So even though \cite[Lemma 2.2]{kropholler} is formulated over $\mathbb{Z}$, one easily verifies that the proof carries through over any field, in particular over $\mathbb{Q}$. It follows that $G$ is a group with $\mathrm{cd}_{\mathbb{Q}}(G)=2$ that is of type $FP_2$ over $\mathbb{Q}$ and satisfies $\mathrm{H}^1(G,\mathbb{Q}[G])=0$. This implies that $G$ is a $2$-dimensional duality group over $\mathbb{Q}$ with dualising module $\mathrm{H}^2(G,\mathbb{Q}[G])\neq 0$.
\end{proof}

 Recall that an $R$-module $M$ is called torsion-free if $r\cdot m =0$ for some $m \in M$ and some non-zero $r\in R$ implies that $m=0$. Note that $\mathbb{Q}[H]$ is an integral domain and that $\mathbb{Q}[G]$ is a torsion-free $\mathbb{Q}[H]$-module.

\begin{lemma}\cite[Lemma 2.3]{kropholler} One has $\tilde{e}(G,H)\geq 2$.
	
\end{lemma}

\begin{proof}
In the proof of \cite[Lemma 2.2]{kropholler}, it is shown that the set of non-zero elements of $\mathbb{Q}[H]$ form a right Ore set in $\mathbb{Q}[G]$. Meaning that for every $a \in \mathbb{Q}[G]$ and $b \in \mathbb{Q}[H]$, there exists an $a' \in \mathbb{Q}[G]$ and $b' \in \mathbb{Q}[H]$ such that $ab'=ba'$. Using the anti-homomorphism
\[ s: \mathbb{Q}[G] \rightarrow \mathbb{Q}[G] : \sum_{g \in G}n_g g  \mapsto  \sum_{g \in G}n_g g^{-1} \]
one easily checks that the set of non-zero elements of $\mathbb{Q}[H]$ also form a left Ore set in $\mathbb{Q}[G]$. Meaning that for every $a \in \mathbb{Q}[G]$ and $b \in \mathbb{Q}[H]$, there exists an $a' \in \mathbb{Q}[G]$ and $b' \in \mathbb{Q}[H]$ such that $b'a=a'b$.	
Let $K$ be the field of fractions of $\mathbb{Q}[H]$ and consider the right $\mathbb{Q}[G]$-module $K\otimes_{\mathbb{Q}[H]} \mathbb{Q}[G]$. Using the left Ore condition, one checks that $K\otimes_{\mathbb{Q}[H]} \mathbb{Q}[G]$ is divisible as a right $H$-module. Since $K\otimes_{\mathbb{Q}[H]} \mathbb{Q}[G]$ is also torsion-free as a right $H$-module it follows from \cite[Lemma 1.2]{kropholler} that
$ \mathrm{H}^2(G,K\otimes_{\mathbb{Q}[H]} \mathbb{Q}[G])=0 $. Since $G$ is of type $FP_2$ over $\mathbb{Q}$, this implies that
\[  K\otimes_{\mathbb{Q}[H]} \mathrm{H}^2(G,\mathbb{Q}[G])=0 .    \] 
 Let $d$ be a non-zero element of $\mathrm{H}^2(G,\mathbb{Q}[G])$ and let $M$ be the left $\mathbb{Q}[H]$-submodule of $\mathrm{H}^2(G,\mathbb{Q}[G])$ generated by $d$. Let $t$ be a generator of $H$. Since $K \otimes_{\mathbb{Q}[H]} M =0$, there exists
an element $a=\sum_{i=0}^k q_i t^i \in \mathbb{Q}[H]$ such that $a\cdot d=0$. This implies that $\{d,t\cdot d,\ldots,t^{k-1}\cdot d \}$ is a generating set for $M$ as a $\mathbb{Q}$-vector space. In particular $M$ a non-zero $\mathbb{Q}[H]$-submodule of $\mathrm{H}^2(G,\mathbb{Q}[G])$ that is finite dimensional as a $\mathbb{Q}$-vector space.

By Corollary \ref{cor: rat end}, we know that $\tilde{e}(G,H)=\dim_{\mathbb{Q}} \mathrm{H}^1(G,\mathcal{F}_H^G(\mathbb{Q}))+1$. Hence, to prove the lemma we need to show that $\mathrm{H}^1(G,\mathcal{F}_H^G(\mathbb{Q}))\neq 0$. Since $G$ is a $2$-dimensional duality group over $\mathbb{Q}$, we have 
\[   \mathrm{H}^1(G,\mathcal{F}_H^G(\mathbb{Q}))\cong \mathrm{H}_1(G,\mathcal{F}_H^G(\mathbb{Q})\otimes_{\mathbb{Q}} D_G)   \]
where $D_G=\mathrm{H}^2(G,\mathbb{Q}[G])$. Since $\mathcal{F}_H^G(\mathbb{Q})\otimes_{\mathbb{Q}} D_G \cong \mathrm{ind}_H^G(\mathrm{coind}_e^H (\mathbb{Q}) \otimes_{\mathbb{Q}} D_G )$, it follows from Shapiro's lemma and Poincar\'{e} duality for $H$ that 
\[   \mathrm{H}^1(G,\mathcal{F}_H^G(\mathbb{Q}))=\mathrm{H}_1(H,\mathrm{coind}_e^H (\mathbb{Q}) \otimes_{\mathbb{Q}} D_G)= (\mathrm{coind}_e^H (\mathbb{Q}) \otimes_{\mathbb{Q}} D_G)^H.   \]
It follows that in order to prove the lemma, it suffices to shows that $(\mathrm{coind}_e^H (\mathbb{Q}) \otimes_{\mathbb{Q}} M)^H\neq 0$, where $M$ is the module constructed above. Since $M$ is finite dimensional as a $\mathbb{Q}$-vector space, there is an isomorphism of left $\mathbb{Q}[H]$-modules
\begin{equation}\label{eq: coind iso} \mathrm{Hom}_{\mathbb{Q}}(\mathbb{Q}[H],\mathbb{Q}) \otimes_{\mathbb{Q}} M \xrightarrow{\cong} \mathrm{Hom}_{\mathbb{Q}}(\mathbb{Q}[H],M): f\otimes m \mapsto \tilde{f}    \end{equation}
where $\tilde{f}(x)=f(x)m$ for all $x\in \mathbb{Q}[H]$. Here, the action of $h \in H$ on an element $\varphi \in \mathrm{Hom}_{\mathbb{Q}}(\mathbb{Q}[H],M) $ is given by $(h\cdot \ \varphi)(x)=h\varphi(xh)$ for all $x \in \mathbb{Q}[H]$. There is also an isomorphism of left $H$-modules
\[\mathrm{Hom}_{\mathbb{Q}}(\mathbb{Q}[H],M) \xrightarrow{\cong} \mathrm{Hom}_{\mathbb{Q}}(\mathbb{Q}[H],M) : f \mapsto \tilde{f}  \]
where $\tilde{f}(h)=hf(h)$ for every $h \in H$. Here the $H$-action on the left hand side is the one appearing on the right hand side of (\ref{eq: coind iso}), while on the right hand side $M$ has the trivial $H$-action. So the right hand side is $\mathrm{coind}_e^H( \mathrm{res}_e^H (M))$. It follows that 
\[  (\mathrm{coind}_e^H (\mathbb{Q}) \otimes_{\mathbb{Q}} M)^H \cong (\mathrm{coind}_e^H(\mathrm{res}^H_e(M)))^H \cong M \neq 0.  \]
\end{proof}

We can now prove Theorem \ref{th: kropholler}.
\begin{proof}
	Since $\tilde{e}(G,H)\geq 2$ by the lemma above, it follows from Lemma \ref{lemma : almost invariant} that $G$ has a proper $H$-almost invariant subset $B$ such that $BH=B$. The proof of \cite[Lemma 3.2]{kropholler} now carries through verbatim showing that $G$ splits over a virtually cyclic subgroup $C$ commensurable to $H$. This means that $G$ is either equal to a non-trivial amalgamated free product $K{\ast}_C L$ or an HNN-extension $G=K\ast_C$ for some subgroups $K,L$ of $G$. First suppose that $G=K{\ast}_C L$. 
	Note that $C$ cannot equal $K$ or $L$, since the amalgamated free product is non-trivial.
	If $C$ has index two in both $K$ and $L$, then $G$ fits into a short exact sequence
	
	\[   1 \rightarrow C \rightarrow G \rightarrow K/C\ast L/C \rightarrow 1.  \]
	Since $K/C\ast L/C$ is isomorphic to the infinite dihedral group and $C$ is infinite virtually cyclic, it is a standard exercise in group theory to check that $G$ contains $\mathbb{Z}^2=\mathrm{BS}(1,1)$. If $C$ has index at least three in either $K$ or $L$ then it follows from the Normal Form Theorem for free amalgamated products that $G$ contains $F_2$. Secondly, suppose that $G=K\ast_{C}$. Again there are two cases two consider. If $K=C$  then the HNN extension is ascending and determined by an injective homomorphism $\varphi: K \rightarrow K$. Since in this case $K$ is infinite virtually cyclic, it is not hard to see that one can find an infinite cyclic subgroup $\langle x \rangle $ of $K$ such that $\varphi$ restricts to $\langle x \rangle$.  Britton's lemma now shows that the subgroup of the ascending HNN extension $G$ generated by $x$ and the stable letter $t$ is isomorphic to a Baumslag-Solitar group $\mathrm{BS}(1,m)$, for some $m \in \mathbb{Z}\setminus \{0\}$.  On other hand, if $K\neq C$ then it follows from Britton's Lemma that $G$ contains $F_2$. This finishes the proof.
	
\end{proof}
We now turn to the proof of Corollary 3 from the introduction.

\begin{proof} Let $G$ be a group satisfying  $\underline{\mathrm{cd}}(G)=2$. It follows from Proposition \ref{prop: long exact} that in order to prove that $\underline{\underline{\mathrm{cd}}}(G)\leq 3$, if suffices to prove that
	$\mathrm{cd}_{\mathcal{F}[H]}(\mathrm{N}_G[H])\leq 3$, where $H$ is any infinite cyclic subgroup of $G$. By \cite[Lemma 3.1]{DP2}, it suffices to show that $\mathrm{cd}_{\mathcal{F}[H]\cap K}(K)\leq 2$, where $K$ is any finitely generated subgroup of $\mathrm{N}_G[H]$ containing $H$.

	 First suppose that $G$ does not contain $F_2$. In this case it follows from the proof of Theorem  \ref{th: kropholler} that $K$ is either an amalgam or an HNN extension where all edge and vertex groups are infinite virtually cyclic subgroups commensurable to $H$. In any case, all the point stabilizers of the action of $K$  on the associated Bass-Serre tree $T$ are commensurable to $H$, implying that $T$ is a $1$-dimensional model for $E_{\mathcal{F}[H]\cap K}K$. We conclude that $\mathrm{cd}_{\mathcal{F}[H]\cap K}(K)\leq 1\leq 2$, as desired.
	
	Next assume that $G$ is coherent, implying that $K$ is finitely presented (actually we only need $K$ to be $FP_2$ over $\mathbb{Z}$). In this case the arguments of the proof of  \cite[Theorem C]{kropholler} carry through completely: using the finite presentability (\cite[Th. 4.1 and below]{kropholler}) one shows that iterating Theorem \ref{th: kropholler} must eventually produce as splitting over virtually cyclic subgroups, proving that $K$ acts simplicially on a tree $T$ with infinite virtually cyclic vertex groups such that each edge stabilizer is commensurable to $H$. 
	Again, we conclude that $T$ is a model for $E_{\mathcal{F}[H]\cap K}K$ and  hence $\mathrm{cd}_{\mathcal{F}[H]\cap K}(K)\leq 1\leq 2$, as desired.
	
	Finally, suppose that $G$ is torsion-free. It this case it follows from \cite[Lemma 2.2]{kropholler} that $K$ is of type $FP_2$ over $\mathbb{Z}$, so the argument of the previous case applies.

	Now suppose $G$ is a one-relator group. Then $\underline{\mathrm{cd}}(G)\leq 2$ by \cite[4.12]{Luck}. If $G$ is torsion-free, then $\underline{\underline{\mathrm{cd}}}(G)\leq 3$ by the above. If $G$ has torsion, then $G$ is word-hyperbolic (e.g.~see \cite[p.~205]{LS}) and \cite[Prop. 9]{LearyPineda} implies that $\underline{\underline{\mathrm{cd}}}(G)\leq 2$.
	
\end{proof}

\begin{remark} \label{remark} \rm
	It follows from the proof above that in order to show that every countable group with $\underline{\mathrm{cd}}(G)=2$ satisfies $\underline{\underline{\mathrm{cd}}}(G)\leq 3$, it suffices to show that if $\Gamma$ is a finitely generated group with $\underline{\mathrm{cd}}(\Gamma)=2$ that has an infinite virtually cyclic subgroup $H$ such that $\mathrm{Comm}_{\Gamma}(H)=\Gamma$, then $\Gamma$ is of type $FP_2$ over $\mathbb{Z}$. If $\Gamma$ is not torsion-free, we are only able to prove that $\Gamma$ is of type $FP_2$ over $\mathbb{Q}$.

\end{remark}

\section{Proof of the Main Theorem}

Let $G$ be a finitely generated group with $\underline{\underline{\mathrm{cd}}}(G)\leq 1$. Our aim is to show that $G$ is virtually cyclic, i.e.~$\underline{\underline{\mathrm{cd}}}(G)=0$. We will argue by contradiction. To this end, assume that $\underline{\underline{\mathrm{cd}}}(G)= 1$. It follows from Lemma \ref{lemma: ration dim} that $\mathrm{cd}_{\mathbb{Q}}(G)=2$ and 
$\mathrm{H}^1(G,F)=0$ for any free $\mathbb{Q}[G]$-module $F$. Moreover, by Corollary \ref{cor: comm loc cyc} the commensurator $\mathrm{N}_G[H]$ of any infinite virtually cyclic subgroup $H$ of $G$ is locally virtually cyclic.
This implies that the family $\mathcal{F}[H]$ of $\mathrm{N}_G[H]$ coincides with the family $\mathcal{VC}$, hence  $\mathrm{cd}_{\mathcal{F}[H]}(\mathrm{N}_G[H])\leq 1$. It follows from Lemma \ref{lem: basic}(3) that 
\[  \mathrm{H}^2_{\mathcal{F}}(G,M) \xrightarrow{\cong}  \prod_{[H]\in \mathcal{I}}\mathrm{H}^2_{\mathcal{F}}(\mathrm{N}_G[H],M).\]
for any right $\mathcal{O}_{\mathcal{VC}}G$-module $M$. This holds in particular for every fixed point functor $\underline{V}$, where $V$ is a $\mathbb{Q}[G]$-module. Let $\mathcal{I}_{0}$ be the subset of $\mathcal{I}$ containing those $[H] \in \mathcal{I}$ for which $\mathrm{N}_G[H]$ is not virtually cyclic. Since finite groups have rational cohomological dimension zero, we conclude from the above and Lemma \ref{lem: fixed point} that

\begin{equation}\label{iso}    \mathrm{H}^2(G,V) \xrightarrow{\cong}  \prod_{[H]\in \mathcal{I}_0}\mathrm{H}^2(\mathrm{N}_G[H],V).\end{equation}
for every $\mathbb{Q}[G]$-module $V$. Since $\mathrm{cd}_{\mathbb{Q}}(G)=2$, we deduce that $\mathcal{I}_0$ is non-empty. Fix $[H]\in \mathcal{I}_0$, denote $N_G[H_0]=N$ and recall that $\mathrm{H}^1(\mathrm{N}_G[H],F)=0$ for any free $\mathbb{Q}[\mathrm{N}_G[H]]$-module $F$ (e.g.~see \cite[Lemma 1.1]{kropholler}) and any $[H]\in \mathcal{I}_0$.  Now consider the short exact sequence of $N$-modules
\begin{equation} \label{seq1}  0 \rightarrow I \rightarrow \mathbb{Q}[N] \xrightarrow{\varepsilon} \mathbb{Q} \rightarrow 0, \end{equation}
where $\varepsilon$ is the augmentation map. Since $\mathbb{Q}[N]^N=\mathrm{H}^1(N,\mathbb{Q}[N])=0$, the long exact cohomology sequence associated to this short exact sequence implies that $\mathrm{H}^1(N,I)\cong \mathbb{Q}$. Inducting (\ref{seq1}) from $N$ to $G$, we obtain a short exact sequence of $G$-modules
\begin{equation} \label{seq2}  0 \rightarrow I_N^G \rightarrow \mathbb{Q}[G] \xrightarrow{\varepsilon} \mathbb{Q}[G/N] \rightarrow 0.  \end{equation}
Since $\mathbb{Q}[G]^G=\mathrm{H}^1(G,\mathbb{Q}[G])=0$, the long exact cohomology sequence associated to this short exact sequence implies that $\mathrm{H}^1(G,I_N^G)=\mathbb{Q}[G/N]^G$. Since $\mathrm{cd}_{\mathbb{Q}}(N)\leq  2$, there exists a short exact sequence of $N$-modules
\begin{equation} \label{seq3}    0 \rightarrow \tilde{F}_2 \rightarrow \tilde{F}_1 \rightarrow I \rightarrow 0,  \end{equation}
where $\tilde{F}_1$ and $\tilde{F}_2$ are free $\mathbb{Q}[N]$-modules.
Inducting (\ref{seq3}) from $N$ to $G$, we obtain a short exact sequence of $G$-modules
\begin{equation} \label{seq4}    0 \rightarrow F_2 \rightarrow F_1 \rightarrow I_N^G \rightarrow 0,   \end{equation}
where $F_1$ and $F_2$ are free $\mathbb{Q}[G]$-modules.
Since $\mathrm{H}^1(\mathrm{N}_G[H], F_1)=\mathrm{H}^1(G, F_1)=0$ for every $[H]\in \mathcal{I}_0$, the long exact cohomology sequence associated to (\ref{seq4}) yields the exact sequences
\[  0 \rightarrow \mathrm{H}^1(G,I_N^G) \rightarrow \mathrm{H}^2(G, F_2) \rightarrow \mathrm{H}^2(G, F_1)    \]
and
\[  0 \rightarrow \mathrm{H}^1(\mathrm{N}_G[H],I_N^G) \rightarrow \mathrm{H}^2(\mathrm{N}_G[H], F_2) \rightarrow \mathrm{H}^2(\mathrm{N}_G[H], F_1)    \]
for every $[H]\in \mathcal{I}_0$. By (\ref{iso}), we obtain a commutative diagram with exact rows.

\[\xymatrix{	0 \ar[r] &  \mathrm{H}^1(G,I_N^G)  \ar[r] \ar[d] & \mathrm{H}^2(G, F_2) \ar[r] \ar[d]^{\cong} & \mathrm{H}^2(G, F_1) \ar[d]^{\cong} \\
	0 \ar[r] & \prod_{[H]\in \mathcal{I}_0} \mathrm{H}^1(\mathrm{N}_G[H],I_N^G)  \ar[r]  & \prod_{[H]\in \mathcal{I}_0}\mathrm{H}^2(\mathrm{N}_G[H], F_2) \ar[r]  & \prod_{[H]\in \mathcal{I}_0}\mathrm{H}^2(\mathrm{N}_G[H], F_1). }\]
We conclude that

\[   \mathbb{Q}[G/N]^G\cong \mathrm{H}^1(G,I_N^G)\cong  \prod_{[H]\in \mathcal{I}_0} \mathrm{H}^1(\mathrm{N}_G[H],I_N^G) .\]

Since $I$ is a direct summand of $I_N^G$, as $N$-modules, $\mathrm{H}^1(N,I_N^G)$ has $\mathrm{H}^1(N,I)\cong\mathbb{Q} $ as a direct summand. We conclude that \[ \mathbb{Q}[G/N]^G \neq 0,\] which implies that $N$ has finite index in $G$. Since $G$ is finitely generated but $N$ is not, this is a contradiction. We conclude that $\underline{\underline{\mathrm{cd}}}(G)= 0$, i.e~$G$ is virtually cyclic.

The Main Theorem now follows easily: if $G$ is a group with  $\underline{\underline{\mathrm{cd}}}(G)\leq 1$, then every finitely generated subgroup $K$ of $G$ also satisfies $\underline{\underline{\mathrm{cd}}}(K)\leq 1$ and is therefore virtually cyclic by the above. This proves that $G$ is locally virtually cyclic. 
\begin{flushright}
$\qed$
\end{flushright}

\section*{Acknowledgment} The author was partially supported by the Danish National Research Foundation through the Centre for Symmetry and Deformation (DNRF92).
The author is also grateful to the Centro de Ciencias Matem\'{a}ticas, UNAM, Morelia, and in particular to No\'{e} B\'{a}rcenas, for their hospitality during a visit when part of this work was completed.


\begin{thebibliography}{12345}
\bibitem{BradyLearyNucinkis} Brady, N., Leary I. , and Nucinkis, B..
{\em On algebraic and geometric dimensions for groups with torsion},  J. Lond. Math. Soc.  Vol 64(2) (2001),  489--500
\bibitem{Bredon} Bredon, G.E.,
{\em Equivalent cohomology theories}, Lecture Notes in Mathematics $34$, Springer ($1967$)
\bibitem{DP1} Degrijse, D. and Petrosyan, N.,
{\em Commensurators and classifying spaces  with virtually cyclic stabilizers}, Groups, Geometry, and Dynamics Vol. 7(3)  (2013), 543--555
\bibitem{DP} Degrijse, D. and Petrosyan, N.,
{\em Geometric dimension of groups for the family of virtually cyclic subgroups},  Journal of Topology (2014) 7(3) 697--726
\bibitem{DP2} Degrijse, D. and Petrosyan, N.,
{\em Bredon cohomological dimensions for groups acting on CAT(0)-spaces}, Groups, Geometry and Dynamics,  9(4) (2015), 1231--€"1265
\bibitem{DP3} Degrijse, D. and Petrosyan, N.,
{\em F-structures and Bredon-Galois cohomology},  Appl. Categor. Struct. 21 (2013), 545--586
 \bibitem{DemPetTal} Dembegioti, F., Petrosyan, N. and Talelli, O.,
 {\em Intermediaries in Bredon (Co)homology and Classifying Spaces}, Publicacions Matem\`{a}tiques 56 (2) (2012), 393--412
\bibitem{Dunwoody} Dunwoody, M.J.,
{Accessibility and groups of cohomological dimension one}, London Math. Soc. (3) 38 (1979), 193--215
\bibitem{EilenbergGanea} Eilenberg S. and Ganea, T.
 {\em On the Lusternik-Schnirelmann category of abstract
 	groups},  Ann. of Math. (2) 65 (1957), 517--518
\bibitem{FluchThesis} Fluch M., 
{\em On Bredon (co-)homological dimensions of groups,} Ph.D. Thesis, University of Southampton, arXiv:$1009.4633$
\bibitem{FluchLeary} Fluch, M. and Leary, I.J.,
{\em An Eilenberg-Ganea Phenomenon for Actions with Virtually Cyclic Stabilizers}, Groups, Geometry, and Dynamics 8(1) (2014), 135--142
\bibitem{FluchNucinkis} Fluch, M., and Nucinkis, B. E. A, 
\emph{On the classifying space for the family of virtually
	cyclic subgroups for elementary amenable groups}, Proc. Amer. Math. Soc. 141(11) (2013), 3755--3769
\bibitem{kropholler} Kropholler, P.H.,
{\em Baumslag-Solitar groups and some other groups of cohomological dimension two}, Commentarii Mathematici Helvetici, Vol. 65(1) (1990), 547--558
\bibitem{KrophollerRoller} Kropholler P.H., and Roller M.A.,
 {\em Relative ends and duality groups}, Journal of Pure and Applied Algebra Vol. 61 (1989),  197--210
\bibitem{LearyPineda} Juan-Pineda, D., and Leary, I. J.,
{\em On classifying spaces for the family of virtually cyclic subgroups}, Recent developments in alg. top., Vol. 407 Contemp. Math., Am. Math. Soc. (2006), 135--145
\bibitem{Luck} L\"{u}ck, W.,   
	{\em Survey on classifying spaces for families of subgroups}, Infinite Groups: Geometric, Combinatorial and Dynamical Aspects, Springer (2005), 269--322
	 \bibitem{LuckL2Book}  L\"{u}ck, W.,
 {\em L2-invariants: theory and applications to geometry and K-theory}, Vol 44 , A Series of Modern Surveys in Mathematics. Springer-Verlag, Berlin, 2002
    
 \bibitem{LuckMeintrup} L\"{u}ck, W. and Meintrup, D.,
{\em On the universal space for group actions with compact isotropy}, Proc. of the conference ``Geometry and Topology'' in Aarhus, (1998), 293--305
\bibitem{LuckReich}  L\"{u}ck, W. and Reich H.,
{\em The Baum-Connes and the Farrell-Jones conjectures in K- and L-theory}, Handbook of K-theory. Vol. 2, Springer, Berlin (2005), 703--842
 \bibitem{LuckWeiermann} L\"{u}ck, W.,  and Weiermann, M.,
 {\em On the classifying space of the family of virtually cyclic subgroups}, Pure and Applied Mathematics Quarterly, Vol. 8(2)  (2012), 497-555
\bibitem {LS} Lyndon, R. and Schupp, P.,
{\em Combinatorial group theory}, Springer (1977)
\bibitem{Martinez} Mart\'{\i}nez-P\'{e}rez, C.,
    {\em A spectral sequence in Bredon (co)homology}, J. Pure Appl. Algebra 176 (2002), 161--173.
 \bibitem{nucinkis} Nucinkis, B.,
 {\em On dimensions in Bredon homology}, Homology, Homotopy and Applications Vol. 6(1) (2004),  33--47
 \bibitem{SerreTrees} Serre, J.-P.,
 {\em Trees}, Springer (1980).
\bibitem{Stallings} Stallings J.R,
{\em  Groups of dimension 1 are locally free},  Bull. Amer. Math. Soc. 74 (1968), 361--364
\bibitem{Swan} Swan, R.G,
 {\em Groups of cohomological dimension one}, J. Algebra 12 (1969), 585--610
\bibitem{Symonds} Symonds, P.,
{\em The Bredon cohomology of subgroup complexes}, J. Pure Appl. Algebra
199 (2005), 261--298





\end{thebibliography}
\end{document}